%% file: root.tex
\documentclass[letterpaper, 10pt, conference,dvipsname]{ieeeconf}
\IEEEoverridecommandlockouts                              

\usepackage{cite}
\let\labelindent\relax
\usepackage{enumitem}

\usepackage{times}
\usepackage{soul}
\usepackage{url}
\usepackage[hidelinks]{hyperref}
\usepackage[utf8]{inputenc}
\usepackage[small]{caption}
\usepackage{graphicx}
\usepackage{booktabs}
\usepackage{algorithm, algpseudocode}
\usepackage[switch]{lineno}
\usepackage{amsmath,bbold,amsthm}
\usepackage{xcolor}
\input{defs}

\usepackage{amssymb}
\usepackage{diagbox}
\usepackage{tikz}
\usetikzlibrary{arrows.meta,calc}

\usepackage{subcaption}
\usepackage{times,bm}
\usepackage{booktabs}
\usepackage[hidelinks]{hyperref}
\allowdisplaybreaks[2]
\begin{document}
\title{Linear Occupancy Homomorphisms for
Discounted Markov Decision Processes}
\author{Yihan Liu, Sarah H.Q. Li\thanks{Y. Liu and S.H.Q. Li are with the Daniel Guggenheim School of Aerospace Engineering, Georgia Institute of Technology, Atlanta, USA, 30332 (email: \{yliu3470,sarahli\}@gatech.edu).}
} %

\markboth{Journal of \LaTeX\ Class Files,~Vol.~14, No.~8, August~2015}%
{Shell \MakeLowercase{\textit{et al.}}: Bare Demo of IEEEtran.cls for IEEE Journals}

\maketitle
\thispagestyle{empty}
\pagestyle{empty}

\begin{abstract}
We introduce \emph{linear occupancy homomorphisms (LOHs)}, a class of Markov decision process (MDP) transformations  that defines the MDP homomorphism over the feasible occupancy measure space of MDPs. In contrast to classic MDP homomorphism frameworks, which are defined through mappings over the state and state-action spaces, LOHs map between occupancy measure spaces and are representable as linear transformation matrices that expose important structural properties of the homomorphism itself. For this class of MDP transformations, we derive finite dimensional linear inequality constraints that certify when a candidate transformation matrix is a valid LOH and derive the corresponding transformed MDP. 
We further establish a connection between LOHs, soft MDP homomorphisms, and factored matrix robust MDPs. When the linear transformation matrix is invertible, we prove that the original MDP is equivalent to a constrained MDP after the homomorphic transformation, and derive necessary and sufficient conditions under which the LOH preserves the optimal value. We illustrate the benefits of LOHs in two settings: improving robust MDP solutions by transforming nonrectangular parameter uncertainty sets into homomorphic rectangular MDPs, and decomposing a large MDP into smaller independent MDPs.
\end{abstract}
\input{1_intro}

\input{2_related_research}

\input{3_prelim}

\input{4_theory_results}
\input{5_simulation.tex}

\input{6_future_work}

\bibliographystyle{IEEEtran}
\bibliography{reference}
\end{document}

%% file: defs.tex
\newcommand{\ones}{\mathbb 1}
\newcommand{\reals}{{\mathbb{R}}}

\newcommand{\argmax}{\mathop{\rm argmax}}

\newcommand{\mnorm}[1]{{\left\vert\kern-0.25ex\left\vert\kern-0.25ex\left\vert #1 
    \right\vert\kern-0.25ex\right\vert\kern-0.25ex\right\vert}}

\newcommand{\mc}{\mathcal}

\newtheorem{definition}{Definition} 
\newtheorem{theorem}{Theorem}

\newtheorem{corollary}{Corollary}
\newtheorem{remark}{Remark}
\newtheorem{proposition}{Proposition}

%% file: 1_intro.tex
\section{Introduction}\label{sec:intro}
Many decision-making problems formulated as MDPs are hard to solve because their structural features are poorly matched to available solution methods.  Examples of this include robust MDP with transition uncertainty correlated across states~\cite{goyal2023beyond,li2026policy,li2025set}, interconnected dynamical networks in which weak coupling prevents subsystems from being solved independently~\cite{bauerle2011mdp,meuleau1998weakly,dean1995decomposition}, or state-action abstractions that obscure exploitable symmetries and decompositions within the underlying dynamics~\cite{lauri2023pomdp,li2006unified,abel2020value,rezaei2022continuous}. Classical MDP homomorphisms~\cite{ravindran2004algebraic,GIVAN2003163,sorg2009transfer} address related challenges by mapping an MDP's state-action space to an alternative representation, typically through state-action pair mappings and aggregation, with the goal of reducing problem complexity or enabling transfer across similar tasks~\cite{sorg2009transfer,abel2020value,rezaei2022continuous}.

However, state aggregations and state-action maps are not the only useful transformations. For example, linear combinations of state-action pairs can define new state-action pairs without simply merging existing pairs. In this paper, we ask whether an MDP homomorphism can be directly defined over the feasible MDP solution space, its occupancy measures, and whether such transformations can expose underlying structure that is difficult to exploit in their original state-action representations. 

This leads us to introduce LOHs, a class of MDP homomorphisms defined as linear transformations of the MDP's occupancy measure spaces. By establishing occupancy measures (instead of state and action spaces) as primitive objects being transformed, LOHs allow the transformation to be designed around \emph{desirable} structure to expose in the transformed MDP. 
Our approach enables us to use MDP homomorphisms as a general tool for restructuring state-action representation while maintaining an explicit mapping between the feasible solutions of the original and transformed problems.

\textbf{Contributions}. We formally define LOH between MDPs with the same state-action cardinality and derive a finite-dimensional certificate that guarantees when a linear matrix is a valid LOH. These conditions lead to an intuitive interpretation of LOH as a simultaneously consistent state and state-action transformations of the occupancy measure spaces, and establish a connection between LOHs, soft MDP homomorphisms, and factored matrix representations of MDP parameter uncertainty. When a valid LOH is invertible, we show that solving the original MDP is equivalent to solving the transformed MDP with an additional cone constraint, and derive necessary and sufficient conditions under which a LOH preserves the optimal value between its original and transformed MDPs. Finally, we use two illustrative numerical examples in robust MDPs and MDP state decompositions to demonstrate that LOHs can be designed to expose computational structure rather than merely reduce model dimension.

%% file: 2_related_research.tex
\section{Related Research}\label{sec:related_research}
MDP homomorphisms transform MDPs while preserving relevant reward and transition structure~\cite{ravindran2004algebraic}. Classical approaches define state-action mappings and equivalence relations that enable model minimization, abstraction, and hierarchical reinforcement learning~\cite{GIVAN2003163,ravindran2004algebraic,sutton2018reinforcement,puterman1994mdp}. Extensions of this research consider approximate homomorphisms and bound the resulting performance loss~\cite{taylor2008bounding}, soft stochastic mappings for transfer between related tasks~\cite{sorg2009transfer,fujimoto2023sale}, and homomorphisms for continuous state-action spaces and policy optimization~\cite{rezaei2022continuous}. In contrast, we define transformations over feasible occupancy-measure spaces. The LOH framework has fundamental connections to robust MDPs and MDP state aggregation and decomposition. 
Standard tractable formulations of robust MDPs rely on rectangular uncertainty sets~\cite{iyengar2005robust,nilim2005robust,wiesemann2013robust}, motivating extensions to coupled and non-rectangular uncertainty~\cite{mannor2016krectangular,goyal2023beyond,li2025set,li2026policy}. Factor-matrix models represent such correlations through shared transition factors~\cite{goyal2023beyond}. LOHs connect naturally to this representation but use occupancy transformations more generally to restructure transition dynamics. MDP state-action abstraction and aggregation reduce decision complexity by exploiting sufficient information states or weak coupling within MDPs~\cite{li2006unified,abel2020value,subramanian2022approximate,dean1995decomposition}. LOH differs in that it exposes block decomposability rather than reducing decision complexity.

%% file: 3_prelim.tex
\section{Preliminaries}
\label{sec:prelim}

\textbf{Notation.} For $n\in\mathbb{N}$, $\reals^n$ is the real vector space of dimension $n$ and $\reals^n_+$ is its nonnegative orthant. Inequalities $(\geq, >)$ between vectors are component-wise. A vector of dimension $k$ where each element is $1$ is denoted by $\ones_k$; $e_j$ denotes a unit vector that is $0$ everywhere and $1$ at the $j^{th}$ element; the cardinality of a set $\mc{X}$ is $|\mc{X}|$; a simplex set over a finite element set $\mc{X}$ is $\Delta(\mc{X}):=\{p\in\reals^{|\mc{X}|}_+:\ones^\top p=1\}$; a linear transformation $T\in\reals^{n\times n}$ over a set  $\mc S\subseteq\reals^n$ is denoted as $T(\mc S):=\{Tx:x\in\mc S\}$; $\operatorname{cone}(T):=\{Tz:z\in\reals^n_+\}$ is the cone generated by the columns of $T$; an indicator function $\mathbf{1}$ takes in a boolean $b$ and outputs $1$ when $b=$ True, and $0$ when $b=$ False.


\subsection{Discounted Infinite Horizon MDP}\label{sec:mdp}
An infinite horizon, discounted MDP is given by $\mathcal{M}(\mathcal{S}, \mathcal{I}, P, \gamma, R, \alpha)$~\cite{puterman1994mdp}. The set $\mc{S}$ is the set of MDP states, and the set $\mathcal{I}$ is the set of all admissible state-action pairs $(s,a)$, with $s\in\mc{S}$. We denote the cardinality of  $\mc{S}$ by $m := |\mc{S}|$, and the cardinality of $\mc{I}$ by $n := |\mc{I}|$.

The MDP transition dynamics are given by matrix $P \in \reals^{m\times n}_+$,
such that each element $P_{s', (s,a)} =P(s'\mid s,a)$ is the conditional probability of transitioning to  $s' \in \mc{S}$ from state-action pair $(s,a)\in \mc{I}$.
A feasible MDP transition matrix $P$ satisfies $\ones^\top_m P = \ones_n^\top$ and is element-wise non-negative.
The MDP reward is given by $R\in\mathbb{R}^{n}$, where each element $R_{s,a}$ is the reward received in state $s$ when taking action $a$. At each timestep $t$, we discount the received state-action reward by a factor $\gamma^t$, where $\gamma \in (0,1)$ is a constant discount factor. We assume that the initial time state is a random variable with state probability distribution given by $\alpha\in\Delta(\mathcal{S})$. Finally, we define the state-action incidence matrix
$D\in\mathbb{R}^{m\times n}$ as
\begin{equation}
    D_{s',(s,a)}
    :=
    \begin{cases}
        1, & s'=s\\
        0, & s'\neq s
    \end{cases}, \forall\  (s,a) \in \mc{I}. 
    \label{eq:D_definition}
\end{equation}
We can represent the infinite-horizon discounted MDP problem via an LP~\cite[Chapter 6.9]{puterman1994mdp}
\begin{equation}
\begin{aligned}
 \max_{x\in\reals^n_+}    & \, R^\top x\\
 \text{s.t. } & (D-\gamma P)x = \alpha,
\end{aligned}\label{eqn:LP_mdp}
\end{equation}
where $x \in\reals^n_+$~\cite[Eq. 6.9.2]{puterman1994mdp} is the expected discounted infinite-horizon occupancy measure in the state-action space, defined as
\begin{equation}\label{eq:sa_op}
x(s,a) =\textstyle\mathbb{E}\!\Big[
\sum_{t=0}^{\infty}
\gamma^t
\mathbf{1}\{s_t=s,\;a_t=a\}
\Big],\, \forall (s,a) \in \mc{I}.
\end{equation}
Given MDP $\mc{M}$, we can characterize the set of all feasible state-action occupancy measures by 
\begin{equation}\label{eqn:feasible_set}
    \mc{F}(\mc{M}) := \bigl\{\, x\in\mathbb{R}^n_+ \;\big|\; (D-\gamma P)\,x=\alpha \,\bigr\}.
\end{equation}
The discounted occupancy measure captures the underlying MDP policy~\cite{puterman1994mdp}, which is a more standard solution for MDPs. We focus on the occupancy measure as the primary solution for an MDP in this paper. Every occupancy measure corresponds to a feasible MDP policy~\cite{puterman1994mdp}.
Additionally,  we  use $d \in \reals_+^m$ to represent the expected discounted infinite-horizon state occupancy measure, defined as
\begin{equation}\label{eq:s_op}
\textstyle d(s)
=
\mathbb{E}\Big[
\sum_{t=0}^{\infty}
\gamma^t
\mathbf{1}\{s_t=s\}
\Big], \quad \forall s \in \mc{S}.
\end{equation}
The occupancy measures $d$ and $x$ jointly  satisfy the discounted flow-conservation constraint~\cite[Chapter 6]{puterman1994mdp}, given by
\begin{equation}\label{eq:flow_constrain_element}
\textstyle d(s')
=
\alpha(s')
+
\gamma
\sum_{s, a}
P(s' \mid s,a)\,x(s,a), \quad \forall s' \in \mc{S},
\end{equation}
which is also the element-wise form of the feasibility constraint in~\eqref{eqn:feasible_set}.
We say that a pair of state and state-action occupancy measures is \textbf{valid} for the MDP $\mc{M}$ if~\eqref{eq:flow_constrain_element} holds.

\section{Linear Occupancy Homomorphisms}




In RL and stochastic optimal control, MDP homomorphisms are often used to transform an MDP into an equivalent MDP that may be easier to solve.  In this section, we introduce an MDP homomorphism using occupancy measure space transformations between MDPs with the same state-action and state cardinalities. 
\begin{definition}[Linear Occupancy Homomorphism (LOH)]\label{def:LOH}
We say that  $T \in \reals^{n\times n}$ is a linear occupancy homomorphism from  $\mathcal{M}(\mathcal{S}, \mathcal{I}, P, \gamma, R, \alpha)$ to $\widehat{\mathcal M}(\hat{\mathcal{S}}, \hat{\mathcal{I}}, \hat{P}, \gamma, \hat{R}, \hat{\alpha})$ if $T$ maps $\mc{M}$'s feasible occupancy measure set into $\widehat{\mc{M}}$'s feasible occupancy measure set, $ T(\mc{F}(\mc{M})) \subseteq \mc{F}(\widehat{\mc{M}})$, or equivalently,
\begin{equation}
   \{Tx | x\geq0, \alpha = Dx - \gamma Px\} \subseteq \{\hat{x} | \hat{x} \geq 0,  \hat{\alpha} = \hat{D} \hat{x} - \gamma \hat{P} \hat{x}\}. 
\end{equation}
If  $ T\big(\mc{F}(\mc{M})\big)= \mc{F}(\widehat{\mc{M}})$, we say that $T$ is a linear occupancy isomorphism between $\mc{M}$ and $\widehat{\mc{M}}$.
\end{definition}

%% file: 4_theory_results.tex
We will discuss in later sections how LOH relates to robust MDPs and soft MDP homomorphisms. For now, we note that based on Definition~\ref{def:LOH}, verifying whether a matrix $T\in\reals^{n\times n}$ is a valid LOH requires evaluating an infinite number of constraints: for all $x \in \reals^{n}_+$  in $\mc{F}(\mc{M})$, verify that $Tx \in \mc{F}(\widehat{\mc{M}})$. To reduce the computational burden of verifying a valid LOH, we show that there exists a set of finite constraints that are sufficient for certifying $T$ as a LOH.
\begin{theorem}[Finite Dimensional LOH Certificate]\label{thm:certificate}
$T$ is a LOH from $\mc{M}(\mathcal{S}, \mc{I}, P, \gamma, R, \alpha)$ to  $\widehat{\mathcal M}(\hat{\mathcal{S}}, \hat{\mc{I}}, \hat{P}, \gamma, \hat{R}, \hat{\alpha})$ if there exists $Y \in \reals^{n\times m}$ and $W\in \reals^{m\times m}$ such that 
\begin{enumerate}
    \item[(A1)] $T-Y(D-\gamma P) \geq 0$, $Y\alpha \geq 0$;
    \item[(A2)] $(\hat{D} - \gamma \hat{P})T = W(D - \gamma P)$, $W\alpha = \hat{\alpha}$.
\end{enumerate}   
\end{theorem}
\begin{proof}
We show that if $Y$ and $W$ exist and satisfy A1 and A2, then $T\mc{F}(\mc{M}) \subseteq \mc{F}(\widehat{\mc{M}})$. For each $x \in \mc{F}(\mc{M})$, $x$ satisfies $x\geq0$ and $(D-\gamma P)x = \alpha$. 
From the first inequality in A1, $Tx \geq Y(D-\gamma P)x = Y\alpha$ whenever $x \geq 0$. From the second inequality in A1, we conclude that $Tx \geq Y\alpha \geq 0$. From the first equality in A2  and $(D-\gamma P)x = \alpha$, we have that $(\hat{D} - \gamma \hat{P})Tx = W(D - \gamma P)x = W\alpha$. From the second equality in A2, we have that $(\hat{D} - \gamma \hat{P})Tx = \hat{\alpha}$. We conclude that $Tx \in \mc{F}(\widehat{\mc{M}})$. 
\end{proof}

Theorem~\ref{thm:certificate} provides a checkable algebraic expression to determine when $T$ forms a LOH. Next, we show that Theorem~\ref{thm:certificate} also provides a method for obtaining the transformed MDP  $\widehat{\mc{M}}$ under LOH $T$. 
\begin{corollary}\label{cor:derive_transformed_mdp}
For an MDP $\mc{M}(\mathcal{S}, \mc{I}, P, \gamma, R, \alpha)$ and a transformation $T$, if there exists $\hat{P} \in \reals_+^{m\times n}$, $W \in \reals_+^{m\times m}$, and $Y \in \reals^{n\times m}$ such that
\[
\hat{P}T = WP, \quad W D = \hat{D}T, \quad 
\ones_m^\top W = \ones_m^\top, \quad \ones_m^\top \hat{P} = \ones_n^\top .\]
and $Y$ satisfies Theorem~\ref{thm:certificate}'s A1, then $T$ is a LOH from $\mc{M}$ to $\widehat{\mathcal M}(\hat{\mathcal{S}}, \hat{\mc{I}}, \hat{P}, \gamma, \hat{R}, W\alpha)$. 
\end{corollary}
\begin{proof}
Since $Y$ exists and satisfies A1 from Theorem~\ref{thm:certificate}, we will show that A2 is also true. We observe that if $WD = \hat{D}T$, and $\hat{P}T = WP$, then the first equality of A2 holds. Additionally, $\hat{P}$ is element-wise non-negative and column stochastic by assumption, so that it is a valid transition matrix. If W is column stochastic and element-wise non-negative, then $\ones^\top W\alpha = \ones^\top\alpha = 1$ and $W\alpha\geq0$, such that $\hat{\alpha} = W\alpha$ is a valid initial state probability distribution. 
\end{proof}
\noindent\textbf{Consistent State and State-action Transformation}.
The construction in Theorem~\ref{thm:certificate} has a natural interpretation of $W$ and $T$ being consistent transformations over the state occupancy and state-action occupancy measures, respectively. Under A2 from Theorem~\ref{thm:certificate}, the flow conservation equation in the transformed MDP $\widehat{\mc{M}}$ becomes
\begin{equation}
    W\alpha = (\hat{D} - \gamma \hat{P})Tx.
\end{equation}
The matrix $T \in \reals^{n\times n}$ transforms the state-action occupancy measure $x$, while the matrix  $W \in \reals^{m\times m}$ transforms the state occupancy measure $d$, producing the occupancy measure space mappings given by
\begin{equation}
\hat{x} = Tx, \quad \hat{d} = Wd, \quad \forall (x,d) \in \mc{F}(\mc{M}).
\end{equation}
The compatibility condition $\hat{D}T = WD$ from Corollary~\ref{cor:derive_transformed_mdp} enforces consistency between the two transformations: aggregating through the transformed state-action occupancy via $\hat{D}$ is identical to aggregating through the original state-occupancy measure via $D$ and then applying the state occupancy transformation $W$. Similarly, the constraint on transition matrices, $\hat{P}T = WP$ ensures that the next state flow induced by the transformed transition matrix $\hat{P}$ is identical to the next state flow induced by the original transition matrix $P$ and then transformed via $W$. Thus, a LOH can be viewed as a pair of
coordinated linear transformations acting on the state-action and state occupancy
spaces that preserve the Markov transition structure of the MDP.

\textbf{Connections to Factor-matrix Parameter Uncertainty}. If $T$ is non-negative and column stochastic, the transition matrix constraint $\hat{P}T = WP$ in Corollary~\ref{cor:derive_transformed_mdp} can be interpreted as a singleton $r$-rectangular uncertainty model from~\cite{goyal2023beyond}, which is used to rectify non-rectangular transition uncertainty in robust MDPs. Specifically, each transformed transition probability vector $WP(\cdot| s,a) = \hat{P}T\big(\cdot, (s,a)\big)$ is equivalent to taking convex combinations of the columns of $\hat{P}$ with weights $T\big(\cdot, (s,a)\big)$. This coincides with $r$-rectangularity when the factor uncertainty set is a singleton. However, we note that LOH and $r$-rectangularity are neither inclusive nor subsets of each other. Corollary~\ref{cor:derive_transformed_mdp} provides only sufficient conditions for LOH feasibility, and it requires an additional requirement $WD = \hat{D}T$, which is not present in $r$-rectangularity.

When the LOH $T$ is invertible and column stochastic, we can directly derive the transformed MDP $\widehat{\mc{M}}$ based on the transformation matrices $(W,T)$. 
\begin{corollary}
\label{cor:constructive_homomorphism}
For an MDP $\mc{M}(\mathcal{S}, \mathcal{I}, P, \gamma, R, \alpha)$, if there exists an invertible and column stochastic matrix ${T}\in\reals_+^{n\times n}$, and a column stochastic matrix  $W\in\reals_+^{m\times m}$, such that 
\begin{enumerate}
    \item $WDT^{-1} \in \{0,1\}^{m\times n}$ is a matrix with $\{0,1\}$ elements;
    \item $WPT^{-1} \geq 0$;
\end{enumerate}
then $T$ is a LOH from $\mc{M}$ to $\widehat{\mathcal M}(\hat{\mathcal{S}}, \hat{\mathcal{I}}, \hat{P}, \gamma, \hat{R}, \hat{\alpha})$ where
\[\hat{P} = WPT^{-1}, \quad \hat{\alpha} = W\alpha, \quad \hat{D} = WDT^{-1}. \]
\end{corollary}
\begin{proof}
From the condition 1),  Corollary~\ref{cor:derive_transformed_mdp}'s assumption $\hat{P}T = WP$ is satisfied, $WD=\hat{D}T$ follows from the condition 2), and $\ones_m^\top WPT^{-1} = \ones_m^\top P T^{-1} = \ones_n^\top T^{-1} = \ones_n^\top$ implies that $\hat{P}$ is column stochastic.  Therefore Corollary~\ref{cor:derive_transformed_mdp} applies.
\end{proof}
\textbf{LOH Synthesis}. Our results primarily characterize and certify candidate LOHs, while a complementary problem is to synthesize a LOH that induces desired structure in the transformed MDP. Such a problem can be posed by searching over $(T,W,Y,\widehat P)$ subject to Theorem~1 and additional structural constraints on $\widehat P$, e.g., rectangular uncertainty or block-diagonal dynamics. We leave this problem for future work and focus on LOH's structural verification in this manuscript.

\noindent\textbf{Connection to Soft MDP Homomorphisms}.
We define a LOH directly as a transformation over the occupancy measure space. This is in contrast to prior MDP homomorphisms~\cite{sorg2009transfer,GIVAN2003163}, which usually define a transformation over the states and state-action spaces directly. For example, a soft MDP homomorphism~\cite{sorg2009transfer} can be extended to a stochastic mapping over state-action pairs. 
\begin{definition}[Soft Homomorphism~\cite{sorg2009transfer}]\label{def:orig_homomorphism_definition}
The functions $\mc{T}: \mc{I} \mapsto \Delta(\hat{\mc{I}})$ and $\mc{W}: \mc{S} \mapsto \Delta(\hat{\mc{S}})$ form a  \emph{soft MDP homomorphism} from $\mathcal{M}:=(\mathcal{S}, \mathcal{I}, P, \gamma, R, \alpha)$ to $\widehat{\mathcal M} :=(\hat{\mathcal{S}}, \hat{\mathcal{I}}, \hat{P}, \gamma, \hat{R}, \hat{\alpha})$ if for all $(s,a) \in \mc{I}$ and $s' \in \mc{S}$, 
  \begin{equation}
     \label{eqn:mdp_homomorphism}
     \textstyle{R}(s,a) =\mc{T}(s,a)^\top \hat{R}, \, \sum_{s' \in \mc{S}}P(s'| s,a)\mc{W}(s') = \hat{P} \mc{T}(s,a).
 \end{equation}
\end{definition}
We can show that under the restricted conditions in Corollary~\ref{cor:constructive_homomorphism}, a LOH induces a linear soft MDP homomorphism.
\begin{proposition}[Induced Soft MDP Homomorphism]
    If $(W, T)$ satisfy the conditions of Corollary~\ref{cor:constructive_homomorphism}, then the functions $\mc{T}: \mc{I} \mapsto \Delta(\hat{\mc{I}})$ and $\mc{W}: \mc{S} \mapsto \Delta(\hat{\mc{S}})$  defined as
    \begin{equation}\label{eqn:homomorphic_def}
        \mc{T}(s,a) := Te_{s,a}, \mc{W}(s) := We_s, \ \forall (s, a) \in \mc{I},
    \end{equation}
form a soft MDP homomorphism from $\mathcal{M}(\mathcal{S}, \mathcal{I}, P, \gamma, R, \alpha)$ to $\widehat{\mathcal M}(\hat{\mathcal{S}}, \hat{\mathcal{I}}, \hat{P}, \gamma, \hat{R}, \hat{\alpha})$ where ${R} = T^{\top}\hat{R}$,  $\hat{P} = WPT^{-1}$, and $\hat{\alpha} = W\alpha$.
\end{proposition}
\begin{proof}
Under~\eqref{eqn:homomorphic_def}, each transformation function $\mc{T}$ and $\mc{W}$ picks out a column of the matrices $T$ and $W$ over $\mc{I}$ and $\mc{S}$. Since $T$ and $W$ are column stochastic, their column vectors are elements of $\Delta(\hat{\mc{I}})$ and $\Delta(\hat{\mc{S}})$, respectively. When $\hat{R} = T^{-\top} R$, all elements $\hat{R}(s,a)$ satisfy the first equality in~\eqref{eqn:mdp_homomorphism}.  From Corollary~\ref{cor:constructive_homomorphism}, the transformed $\widehat{\mc{M}}$ is a feasible MDP, and $\hat{P}T = WP$, which ensures that $P(s'| s,a)$ satisfies~\eqref{eqn:mdp_homomorphism} for all $(s', s,a) \in \mc{S}\times \mc{I}$.  
\end{proof}
\subsection{Geometry and Optimality Preservation}
\label{sec:geometry_optimality}
After forming a LOH, we quantify the optimality gap between the MDP $\mc{M}$ and its homomorphic transformation $\widehat{\mc{M}}$ by solving the MDP in its original occupancy measure space $\mc{F}(\mc{M})$ vs in the transformed occupancy measure space $\mc{F}(\widehat{\mc{M}})$. In this section, we first characterize the set inclusion relationship between the transformed feasible set $\mc{F}(\widehat{\mc{M}})$ and the original feasible set $\mathcal{F}(\mc{M})$ and use this relationship to determine whether an optimality gap may exist. 

Consider the conic hull of LOH $T$: the space spanned by non-negative combinations of its columns, given by
\begin{equation}
    \operatorname{cone}(T)
    :=
    T\mathbb{R}_{+}^{n}
    =
    \{Ty:y\in\mathbb{R}_{+}^{n}\}.
    \label{eq:cone_definition}
\end{equation}
In the LOH definition, we noted that $T\mathcal F(\mc{M})\subseteq\mc{F}(\hat{\mc{M}})$, next we show that $\operatorname{cone}(T)$ is exactly what constrains  $T\mc{F}(\mc{M})$.
\begin{proposition}[Image of the transformed feasible set]\label{prop:feasible_set_geometry}
For a LOH $T$ from $\mathcal{M}$ to $\widehat{\mathcal{M}}$ that satisfies the conditions of Corollary~\ref{cor:constructive_homomorphism},    $T\mathcal F(\mc{M})={\mathcal F}(\widehat{\mc{M}})\cap\operatorname{cone}(T)$.
\end{proposition}
\begin{proof}
First, we show $T\mathcal{F}(\mc{M})\subseteq {\mathcal{F}}(\widehat{\mc{M}})\cap\operatorname{cone}(T).$ By definition, $T\mathcal{F}(\mc{M})\subseteq {\mathcal{F}}(\widehat{\mc{M}})$. Since $\mathcal{F}(\mc{M})$ consists of strictly non-negative vectors, $T\mathcal{F}(\mc{M}) \subseteq \operatorname{cone}(T)$ as well. 

For the reverse inclusion, consider $\hat{x} \in \mc{F}(\widehat{\mc{M}}) \cap \operatorname{cone}(T)$. Since $\hat{x} \in \operatorname{cone}(T)$, there exists $x$ such that $\hat{x} = Tx$. Additionally, $\hat{x} \in \mc{F}(\widehat{\mc{M}})$ implies that $(\hat{D} - \gamma\hat{P})\hat{x} = \hat{\alpha}$. Under Corollary~\ref{cor:constructive_homomorphism}, $ (\hat{D} - \gamma\hat{P})=W ({D} - \gamma{P}) T^{-1}, \, \widehat\alpha=W\alpha.$ 
This implies that $W({D} - \gamma{P})T^{-1}Tx= W\alpha$. To conclude from here that $(D-\gamma P)x  = \alpha$, we require $W$ to be full rank. We note that $W$ satisfies $\hat{D} = WDT^{-1}$. Since $\hat{D}$ is a state-action incidence matrix for $\widehat{M}$, $\text{rank}(\hat{D}) = m$, therefore $W \in \reals^{m\times m}$ is full rank and invertible, such that $({D} - \gamma{P})x = \alpha$. Since $x \geq 0$, $x \in \mc{F}(\mc{M})$. 
\end{proof}


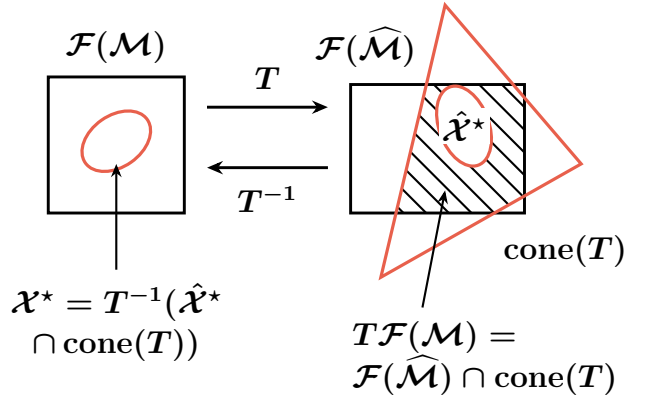
\begin{figure}[t]
\centering
\newcommand{\figfont}{\bfseries\boldmath\fontsize{12pt}{14pt}\selectfont}
\begin{tikzpicture}[
    every node/.style={font=\figfont},
    arr/.style={-{Stealth[length=6pt,width=5pt]}, line width=1.1pt},
    red/.style={draw=figred, line width=1.2pt}
]
\definecolor{figred}{RGB}{232,96,76}

\draw[line width=1.1pt] (0,0) rectangle (1.8,1.8);
\node at (0.9,2.2) {$\mathcal F(\mathcal M)$};
\draw[red, rotate around={35:(0.9,0.95)}] (0.9,0.95) ellipse (0.5 and 0.34);

\begin{scope}[shift={(4.0,0)}]
  \def\W{2.3}\def\H{1.7}
  \draw[line width=1.1pt] (0,0) rectangle (\W,\H);
  \node at (0.2,2.2) {$\mathcal F(\widehat{\mathcal M})$};

  \begin{scope}
    \clip (0,0) rectangle (\W,\H);
    \clip (1.25,2.75) -- (3.05,0.65) -- (0.4,-0.85) -- cycle;
    \foreach \k in {-3,-2.7,...,6} {\draw[line width=0.8pt] (\k,\H+0.5) -- ++(3.4,-3.4);}
  \end{scope}

  \draw[red] (1.25,2.75) -- (3.05,0.65) -- (0.4,-0.85) -- cycle;
  \draw[red, fill=white, rotate around={-70:(1.5,1.15)}] (1.5,1.15) ellipse (0.55 and 0.33);
  \node[fill=white, inner sep=1pt] at (1.5,1.15) {$\hat{\mathcal X}^\star$};
\end{scope}

\draw[arr] (2.1,1.4) -- (3.7,1.4) node[midway, above=3pt] {$T$};
\draw[arr] (3.7,0.6) -- (2.1,0.6) node[midway, below=3pt] {$T^{-1}$};

\node[align=center, anchor=north] at (0.9,-0.8)
  {$\mathcal X^\star=T^{-1}(\hat{\mathcal X}^\star$\\[1pt]$\cap\,\mathrm{cone}(T))$};
\draw[-{Stealth[length=5pt,width=4pt]}, line width=0.8pt] (0.9,-0.75) -- (0.9,0.65);       

\node[align=left, anchor=north west] at (3.9,-1.3)
  {$T\mathcal F(\mathcal M)=$\\[1pt]$\mathcal F(\widehat{\mathcal M})\cap\mathrm{cone}(T)$};
\draw[-{Stealth[length=5pt,width=4pt]}, line width=0.8pt] (4.95,-1.25) -- (5.25,0.3);     

\node[anchor=west] at (5.9,-0.5) {$\mathrm{cone}(T)$};
\end{tikzpicture}
\caption{The regions $\mathcal F(\mathcal M)$, $\mathcal F(\widehat{\mathcal M})$, $\operatorname{cone}(T)$ and $T\mathcal F(\mathcal M)=\mathcal F(\widehat{\mathcal M})\cap\mathcal{\operatorname{cone}}(T)$.}
\label{fig:geometry}
\vspace{-2em}
\end{figure}
\begin{remark} 
\label{rem:geometric_interpretation} Proposition~\ref{prop:feasible_set_geometry} shows that while the transformed MDP's feasible set contains all transformed occupancy measures from $\mc{F}(\mc{M})$, it will contain additional feasible occupancy measures that have no correspondence to the original MDP's occupancy measures. These additional occupancy measures are tangential to $\operatorname{cone}(T)$.
\end{remark}
A key implication of Proposition~\ref{prop:feasible_set_geometry} is that a LOH is \emph{not surjective} from the original occupancy measure space $\mc{F}(\mc{M})$ to the transformed MDP's occupancy measure space $\mc{F}(\widehat{\mc{M}})$. As a result, the transformed MDP has a larger feasible occupancy measure space in general and may achieve a higher objective than the original MDP.

\begin{theorem}[MDP Objective Gap Under LOH]
\label{thm:opt}
Consider a matrix pair $T$ and $W$ which satisfies the assumptions of Corollary~\ref{cor:constructive_homomorphism} and forms a LOH from $\mc{M}(\mathcal{S}, \mathcal{I}, P, \gamma, R, \alpha)$ to $\widehat{\mathcal M}(\hat{\mathcal{S}}, \hat{\mathcal{I}}, \hat{P}, \gamma, \hat{R}, \hat{\alpha})$, where $\hat{R} = T^{-\top} R$. Let $V^\star$ and $\hat{V}^\star$ denote the optimal value of each respective MDP, i.e.,
\begin{equation}
    V^\star
    =
    \max_{x\in\mathcal F(\mc{M})}
    R^\top x, \quad 
    \widehat{V}^\star  = \max_{\widehat x\in\mathcal{F}(\widehat{\mc{M}})}\hat{R}^\top \hat{x},
    \label{eq:restricted_optimization}
\end{equation}
then $V^\star \leq \hat{V}^\star$. 
\end{theorem}
\begin{proof}
The LP form of MDP $\mc{M}$ is $ \max_{x\in \mc{F}(\mc{M})}     R^\top x$. Let $\hat{x} = Tx$ where $T$ is invertible and satisfies the assumptions of Corollary~\ref{cor:constructive_homomorphism} with matrix $W$. It holds that $x \in\mc{F}(\mc{M})$ if and only if $Tx \in T\big(\mc{F}(\mc{M})\big)$, and the optimal value of $\mc{M}$ is equivalently given by 
 \begin{equation}\label{eqn:thm2_pf_1}
 \textstyle V^\star = \max_{x}    (T^{-\top} R)^\top Tx, \text{ s.t. }  Tx \in T\big(\mc{F}(\mc{M})\big). 
\end{equation}
From Proposition~\ref{prop:feasible_set_geometry}, $T\big(\mc{F}(\mc{M})\big) = \mc{F}(\widehat{\mc{M}}) \cap \operatorname{cone}(T)$, and $\hat{R} = T^{-\top} R$. Therefore,
\begin{equation}\label{eqn:thm2_pf_2}
    \textstyle V^\star =\max_{\hat{x}}     \hat{R}^\top\hat{x} \text{ s.t. } \hat{x} \in \mc{F}(\widehat{\mc{M}}) \cap \operatorname{cone}(T). 
\end{equation}
Since~\eqref{eqn:thm2_pf_2} is identical to the second maximization problem in~\eqref{eq:restricted_optimization} except with an additional constraint, we conclude $V^\star \leq \hat{V}^\star$.
\end{proof}
Theorem~\ref{thm:opt} shows that optimizing the transformed MDP with an additional constraint $\operatorname{cone}(T)$ is equivalent to optimizing the original MDP.
We leverage this fact in the following section to show that we can optimize non-rectangular robust MDPs better than rectangular approaches can~\cite{wiesemann2013robust,goyal2023beyond}. Finally, we provide the necessary and sufficient condition to achieve zero optimality gap. 
\begin{corollary}\label{cor:exact_optimality}
Let the set of optimal occupancy measures for $\widehat{\mc{M}}$ be denoted as $\hat{\mc{X}}^\star$, such that 
\[\textstyle \hat{\mc{X}}^\star = \argmax_{\widehat x\in\mathcal{F}(\widehat{\mc{M}})}\hat{R}^\top \hat{x},\]
then the following conditions are equivalent: 
    \begin{enumerate}
        \item $\widehat V^\star=V^\star$;
        \item $\widehat{\mathcal X}^\star\cap\operatorname{cone}(T)\neq\varnothing$;
        \item there exists
$\widehat x^{\star}\in\widehat{\mathcal{X}}^{\star}$ such that $T^{-1}\widehat x^{\star}\geq0,$
    \end{enumerate}
where $V^\star$ and $\widehat{V}^\star$ are the optimal value of the MDPs $\mc{M}$ and $\widehat{\mc{M}}$ under LOH defined in~\eqref{eq:restricted_optimization}. 
\end{corollary}
\begin{proof}
By Theorem~\ref{thm:opt},
$V^\star$ is the optimal value of the transformed MDP $\widehat{\mc{M}}$ with an additional constraint $\operatorname{cone}(T)$. Therefore,
$\widehat V^\star=V^\star$ if and only if
$\mc{F}(\widehat{\mc{M}})  \cap\operatorname{cone}(T) \cap\widehat{\mathcal X}^\star \neq \emptyset$, which proves the equivalence between conditions 1 and 2.

Since $T$ is invertible and
$\operatorname{cone}(T)=\{Tx:x\geq0\}$, we have $\widehat x^\star\in\operatorname{cone}(T)$ if and only if $T^{-1}\widehat x^\star\geq0.$ Hence, conditions 2 and 3 are equivalent.
\end{proof}
In Figure~\ref{fig:geometry}, we visualize the relationship between the original and transformed feasible occupancy spaces in Proposition~\ref{prop:feasible_set_geometry}. Transformation $T$ sends $\mc{F}(\mc{M})$ to the hatched region $\mc F(\widehat{\mc M})\cap\operatorname{cone}(T)$. By Corollary~\ref{cor:exact_optimality}, $\hat V^\star=V^\star$ if and only if there exists optimal occupancy measure of $\widehat{\mc{M}}$ that lies in the hatched region.

%% file: 5_simulation.tex
\section{Numerical examples}
We develop two illustrative examples of LOH to robust MDP and MDP state-action decomposition to demonstrate how LOH improves the optimality and efficiency of existing solution approaches in each of these domains. 
\subsection{Non-rectangular Robust MDP}
A robust MDP~\cite{iyengar2005robust,wiesemann2013robust} seeks to find a policy that optimizes its worst-case discounted return induced by a set of transition matrices, $\mc{U} \subseteq \reals^{m\times n}_+$. For rectangular transition uncertainty sets, the worst-case optimal value formulation is equivalent to the standard robust MDP objective, which can be expressed as
\begin{equation}
    V^\star(\mathcal U)=\min_{P\in\mathcal U}\max_{x \in \mc{F}(P) }R^\top x.
    \label{eq:robust_mdp}
\end{equation}
The solvability of~\eqref{eq:robust_mdp} relies critically on the structure of the transition uncertainty set, $\mc{U}$. When the uncertainty in the MDP transition dynamics is independent across $(s,a)$ pairs ($\mc{U} = \prod_{(s,a) \in \mc{I}} \mc{U}_{s,a}$) or across states ($\mc{U} = \prod_{s \in \mc{S}} \mc{U}_{s}$), value function-based methods such as dynamic programming apply. Beyond these rectangular settings, guarantee of finding~\eqref{eq:robust_mdp}'s optimal solution is sparse (e.g. see $k$-rectangularity~\cite{mannor2016krectangular} and $r$-rectangularity~\cite{goyal2023beyond}). We demonstrate that LOH provides an intuitive and natural way to consider linear transformations over the transition uncertainty set $\mc{U}$ that may lead to more optimal solutions to~\eqref{eq:robust_mdp}.

We consider a robust MDP with  state set $\mathcal S=\{s_1,s_2,s_3\}$, state-action set $\mc{I} = \{(s_1,a_1),(s_1,a_2),(s_2,a_1),(s_3,a_1)\}$ and discount factor $\gamma=0.9$, such that $m = 3$ and $n =4$. In this MDP,  only state $s_1$ has two admissible actions while the remaining states only have one admissible action. The initial distribution is $\alpha=\begin{bmatrix}0.5&0.5&0\end{bmatrix}^\top$ and the reward vector is $R = \begin{bmatrix} 0&0.1&0&0.1 \end{bmatrix}^\top$. 
The transition matrix $P \in \reals^{3\times 4}_+$ is uncertain and parametrized by $q = [q_1,\,q_2,\,q_3]$, given by
\begin{equation}
    P(q)=
    \begin{bmatrix}
        \tfrac{1-q_1}{2} & \tfrac{1-q_2}{2} & \tfrac{1-q_3}{2} & 0\\[0.2em]
        \tfrac{1-q_1}{2} & \tfrac{1-q_2}{2} & \tfrac{1-q_3}{2} & 0\\[0.2em]
        q_1 & q_2 & q_3 & 1
    \end{bmatrix}. \label{eqn:uncertainty_set}
\end{equation}
\textbf{Correlated uncertainty set}. The transition uncertainty set is given by 
\begin{equation}
    \mc{U} = \left\{P(q)\mid q  \in \{\begin{bmatrix}0.9 \\ 0.45\\ 0.1\end{bmatrix}, \begin{bmatrix}0.55 \\ 0.1\\ 0.8\end{bmatrix}\}\right\}. 
\end{equation}
When there are only two possible $q$ values, $\mc{U}$ is a  non-$(s,a)$-rectangular and non-$s$-rectangular transition uncertainty set; knowing any one of $q_1,q_2,q_3$ determines the other two values. We directly solve~\eqref{eq:robust_mdp} under correlated uncertainty set $\mc{U}$ via the HiGHs solver in CVXPY~\cite{huangfu2018parallelizing}, and obtain that $V^\star(\mc{U})= 0.86$, and that the optimal occupancy measure is $x^\star=\begin{bmatrix}0 & 1.4 & 1.4 & 7.1\end{bmatrix}^\top .$

\textbf{Rectangular Over-approximation}. Non-rectangular uncertainty sets in robust MDPs are often over-approximated via rectangular uncertainty sets. We visualize the over-approximation procedure in Figure~\ref{fig:transition_sets}, where we plot the transition probability values $P(s_3| s_2, a_1)$ and $P(s_3| s_1, a_1)$ of each uncertainty set. We observe that the rectangular over-approximation of $\mc{U}$ introduces two pairs of transition probability values that are not part of $\mc{U}$. Numerically, $\mc{U}$'s  $(s,a)$-rectangular approximation set is given by 
\begin{equation}
\mc{U}_r = \left\{P\bigl(\begin{bmatrix}
         q_1\\q_2\\q_3
     \end{bmatrix}\bigr) \;\middle|\;
     \begin{aligned}
       q_1 &\in \{0.55,\,0.9\},\\
       q_2 &\in \{0.1,\,0.45\},\\
       q_3 &\in \{0.1,\,0.8\}
     \end{aligned}\right\}.
\end{equation}
Optimizing the robust MDP~\eqref{eq:robust_mdp} with respect to the rectangular transition uncertainty set $\mc{U}_r$, we obtain $V^\star(\mc{U}_r)=0.75$ and an optimal occupancy measure of $x^\star=\begin{bmatrix}1.3 & 0 & 1.3 & 7.5\end{bmatrix}^\top$. The optimality gap $V^\star(\mathcal U)-V^\star(\mc{U}_r)=0.11$ exists due to the additional transition probability vectors that $\mc{U}_r$ introduced during over-approximation. Once optimized over the worst case, the resulting policy will be more conservative in nature.


\textbf{Occupancy Transformation via LOH}. Instead of over-approximating $\mathcal U$, we consider the following pair of linear state and state-action transformation matrices, 
\begin{equation}
    W=\begin{bmatrix} 2 & -1 & 0\\ -1 & 2 & 0\\ 0 & 0 & 1\end{bmatrix},
    \quad
    T=\begin{bmatrix} 2 & 1 & -1 & 0\\ 0 & 1 & 0 & 0\\ -1 & -1 & 2 & 0\\ 0 & 0 & 0 & 1\end{bmatrix}.
    \label{eq:WT}
\end{equation}
We verify that $T$ and $W$ satisfy Theorem~\ref{thm:certificate}'s conditions (A1) with a valid $Y$ and the equality constraint in A2, therefore, $T$ is a valid LOH. Since this example is a same-size LOH, we have $\hat D=D$ and $WD=\hat DT$. From (A2), the transformed transition matrix is therefore $\hat P=WPT^{-1}$. We can individually transform the two transition matrices in the correlated uncertainty set $\mc{U}$. Let $v_1=\begin{bmatrix}0.18 & 0.18 & 0.64\end{bmatrix}^\top,
v_2=\begin{bmatrix}0.41 & 0.41 & 0.18\end{bmatrix}^\top,
v_3=\begin{bmatrix}0.32 & 0.32 & 0.36\end{bmatrix}^\top,
v_4=\begin{bmatrix}0.14 & 0.14 & 0.72\end{bmatrix}^\top,
v_5=\begin{bmatrix}0 & 0 & 1\end{bmatrix}^\top .$ 
Then the transformed transition uncertainty set $\widehat{\mc{U}}$ is given by

\begin{equation}\label{eqn:transformed_U}
\widehat{\mc{U}}=\left\{
\begin{aligned}
\widehat P^1&:=\begin{bmatrix}v_1 & v_2 & v_3 & v_5\end{bmatrix},\\
\widehat P^2&:=\begin{bmatrix}v_1 & v_2 & v_4 & v_5\end{bmatrix}
\end{aligned}
\right\}.
\end{equation}
Under LOH $T$, we observe that the only transition matrix uncertainty remaining is strictly in a single state-action $(s_2,a_1)$. Hence $\widehat{\mathcal U}$~\eqref{eqn:transformed_U} is $(s,a)$-rectangular.
Solving the transformed rectangular robust MDP~\eqref{eq:robust_mdp} with $\widehat{\mc{U}}$, we obtain  $V^\star(\widehat{\mathcal U})=0.86,$ with optimal occupancy measure $\hat{x}^\star=\begin{bmatrix}0 & 1.4 & 1.4 & 7.1\end{bmatrix}^\top$, which recovers the optimal value $V^\star(\mc{U})$. Overall,
\begin{equation}
    V^\star(\mc{U}_r)=0.75<V^\star(\widehat{\mathcal U}) =V^\star(\mathcal U)=0.86 .
    \label{eq:robust_final_comparison}
\end{equation}
\begin{figure}[t]
    \centering
    \includegraphics[width=0.9\columnwidth]{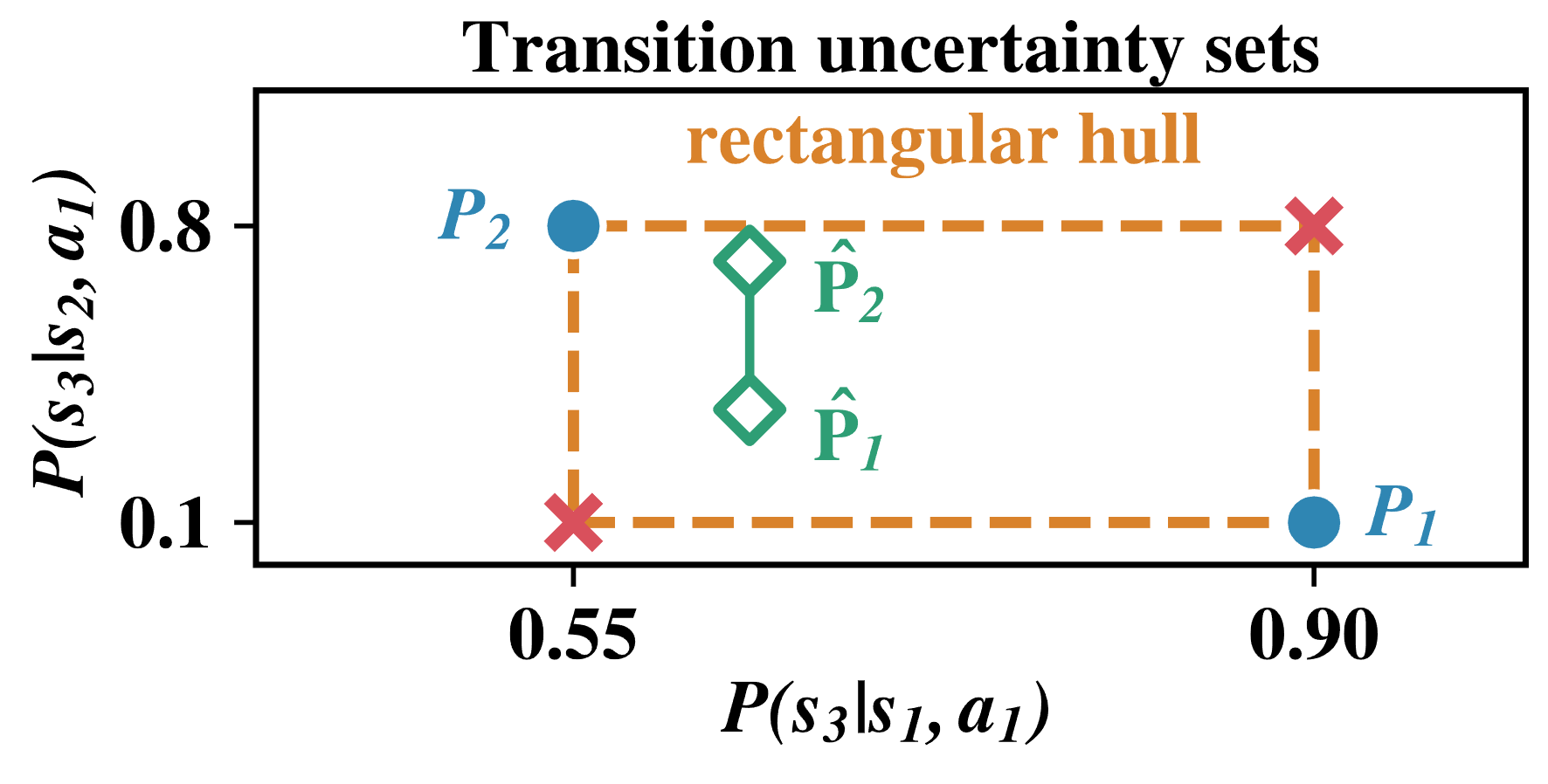}
    \caption{Transition uncertainty sets  $\mc{U}$, $\mc{U}_r$ and $\widehat{\mc{U}}$ projected onto $P(s_3\mid s_1,a_1)$ and $P(s_3\mid s_2,a_1)$. Blue is $\mc{U}$ and is non-rectangular; dashed orange is its rectangular over-approximation $\mc{U}_r$; green is $\widehat{\mc{U}}$ and is rectangular.}
    \vspace{-2em}\label{fig:transition_sets}
\end{figure}

\subsection{MDP State-action Decomposition}
\label{subsec:decoupling_example}
In addition to facilitating better solutions for robust MDPs, LOHs can be used to discover state-action decompositions within an MDP that can reduce an MDP with a large state-action space into multiple MDPs with smaller state-action spaces, which  is more amenable to distributed computing and parallel solver architectures. To achieve this, we seek LOHs such that the transformed MDP's transition matrix $\widehat{P}$ has  block-diagonal structure. 

Consider an MDP with state set $\mathcal S=\{s_1,s_2,s_3,s_4\}$ and state-action set $\mc{I} = \{(s_1,a_1),(s_1,a_2), (s_2,a_1),$ $(s_3,a_1),(s_3,a_2),(s_4,a_1)\}$, where only $s_1$ and $s_3$ have two actions, while $s_2$ and $s_4$ have one action each. The MDP's discount factor is $\gamma=0.9$,  initial state probability distribution is $\alpha=\begin{bmatrix}1&0&0&0\end{bmatrix}^\top$, reward vector is $R=\begin{bmatrix}0&0&0.1&0&0&0.1\end{bmatrix}^\top$, and transition matrix is
\begin{equation}
    P=
    \begin{bmatrix}
        0.20 & 0.60 & 0 & 0    & 0    & 0\\
        0.40 & 0.20 & 1 & 0    & 0    & 0\\
        0.15 & 0.05 & 0 & 0.50 & 0.70 & 0.30\\
        0.25 & 0.15 & 0 & 0.50 & 0.30 & 0.70
    \end{bmatrix}.
\end{equation}
Solving it gives $V^\star=0.71$ with the optimal occupancy measure $x^\star=\begin{bmatrix}1.2&0&4.4&1.6&0&2.7\end{bmatrix}^\top.$

\textbf{Occupancy Transformation via LOH}. We consider the following pair of linear state and state-action transformation matrices  $W \in \reals^{4\times 4}$ and $T \in \reals^{6\times 6}$, given by
\begin{equation}
    W=
    \begin{bmatrix}
        0.5&0&0&0\\
        0&1&0&0\\
        0.5&0&1&0\\
        0&0&0&1
    \end{bmatrix}, T=
    \begin{bmatrix}
        0.5I_2&0&0&0\\
        0&1&0&0\\
        0.5I_2&0&I_2&0\\
        0&0&0&1
    \end{bmatrix},
\end{equation}
where $(W,T)$ satisfy the requirements of Corollary~\ref{cor:constructive_homomorphism}, and therefore produce a LOH from the original MDP to a transformed MDP with initial state probability distribution $\widehat\alpha =W\alpha=\begin{bmatrix}0.5&0&0.5&0\end{bmatrix}^\top$, reward vector $\widehat{R} = T^{-\top} R = \begin{bmatrix}0&0&0.1&0&0&0.1\end{bmatrix}^\top$, and transition matrix $\widehat{P}$ given by
\begin{equation}
    \widehat P=WPT^{-1}=
    \begin{bmatrix}
        0.2&0.6&0&0&0&0\\
        0.8&0.4&1&0&0&0\\
        0&0&0&0.5&0.7&0.3\\
        0&0&0&0.5&0.3&0.7
    \end{bmatrix}.
    \label{eq:decoupled_kernel}
\end{equation}
Here, $\widehat{P}$ is block diagonal. The transformed MDP can thereby be solved via two separate MDPs, each with $2$ states and $3$ state-action pairs. Let $\mc{M}_1$ denote the MDP corresponding to the states $\mc{S}_1 = \{\hat{s}_1, \hat{s}_2\}$ and state-action pairs $\mc{I}_1 = \{(\hat{s}_1, \hat{a}_1), (\hat{s}_1, \hat{a}_2),(\hat{s}_2, \hat{a}_1)\}$, and let $\mc{M}_2$ denote the second MDP corresponding to the states $\mc{S}_2 = \{\hat{s}_3, \hat{s}_4\}$ and state-action pairs $\mc{I}_2 = \{(\hat{s}_3, \hat{a}_1), (\hat{s}_3, \hat{a}_2),(\hat{s}_4, \hat{a}_1)\}$. Additionally, let $\widehat{V}^\star_1$ denote the optimal value of the first MDP and $\widehat{V}^\star_2$ denote the optimal value of the second MDP. Solving the two MDPs separately gives
\begin{equation}
\begin{aligned}
    \widehat V^\star_1=0.44, & \quad \widehat x_1^\star=\begin{bmatrix}0.61&0&4.4\end{bmatrix}^\top,\\
    \widehat V^\star_2=0.27, &\quad\widehat x_2^\star=\begin{bmatrix}2.3&0&2.7\end{bmatrix}^\top.
\end{aligned}
\label{eq:decoupled_values}
\end{equation}
When combined, these values give $\widehat V^\star_1+\widehat V^\star_2 =0.71=V^\star.$ Stacking the two solutions gives the optimal occupancy measure of the transformed MDP, $\widehat{x}^\star=\begin{bmatrix}\widehat{x}_1^{\star\top}&\widehat{x}_2^{\star\top}\end{bmatrix}^\top$, and mapping it back via $T^{-1}$ gives $T^{-1}\widehat{x}^\star=\begin{bmatrix}1.2&0&4.4&1.6&0&2.7\end{bmatrix}^\top=x^\star$, which returns the optimal occupancy measure of the original MDP.

%% file: 6_future_work.tex
\section{Conclusion}
\label{sec:conclusion}
We introduced LOH for discounted MDPs over finite state-action spaces. We provide sufficient conditions for when the transformed  MDP  is homomorphic to the original MDP and when it produces the same optimal value. We demonstrate how the LOH framework is useful to optimize robust MDP under nonrectangular transition uncertainty sets. Future work will involve finding LOHs for these applications over larger state-action spaces and exploring value function-based MDP solution approaches under LOH.